\documentclass{article}

\usepackage{arxiv}

\usepackage[utf8]{inputenc} 
\usepackage[T1]{fontenc}    
\usepackage{hyperref}       
\usepackage{url}            
\usepackage{booktabs}       
\usepackage{amsfonts}       
\usepackage{nicefrac}       
\usepackage{microtype}      
\usepackage{lipsum}
\usepackage{graphicx}
\usepackage{amsmath,amssymb,amsthm}
\usepackage{subcaption}
\usepackage{enumerate}
\usepackage{float}

\newtheorem{theorem}{Theorem}[section]
\newtheorem{lemma}[theorem]{Lemma}
\newtheorem{proposition}{Proposition}
\theoremstyle{definition}
\newtheorem{definition}[theorem]{Definition}

\def\Cx{\mathbb{C}}
\def\Chat{\widehat{\mathbb{C}}}
\def\Isom{\mathcal{I}}
\def\SO3{\text{SO}(3)}

\title{Symmetries for Julia Sets of Rational Maps}

\author{
  Gustavo R.~Ferreira \\
  Institute of Mathematics and Statistics\\
  University of S\~ao Puulo\\
  S\~ao Paulo, SP, Brazil \\
  \texttt{gustavo.rodrigues.ferreira@usp.br} \\
}

\begin{document}
\maketitle

\begin{abstract}
Since the 1980s, much progress has been done in completely determining which functions share a Julia set. The polynomial case was completely solved in 1995, and it was shown that the symmetries of the Julia set play a central role in answering this question. The rational case remains open, but it was already shown to be much more complex than the polynomial one. Here, we offer partial extensions to Beardon's results on the symmetry group of Julia sets, and discuss them in the context of singularly perturbed maps.
\end{abstract}


\section{Introduction}
In complex dynamics, the problem of finding maps with the same Julia set goes back to Julia himself~\cite{Julia}. During the 1980s and early 1990s, a complete description for the polynomial case was obtained through the efforts of Baker, Eremenko, Beardon, Steinmetz and others~\cite{BE87,Beardon90,Beardon92,SS95}. The culmination of this work is the following theorem: given any Julia set $J$ for a non-exceptional polynomial, there exists a polynomial $P$ such that the set of all polynomials with Julia set $J$ is given by
\begin{equation} \label{eq:Stein}
\mathfrak{P}(J) = \{\sigma\circ P^n : n \geq 1 \text{ and } \sigma\in\Sigma_J\},
\end{equation}

where $\Sigma_J$ denotes the set of symmetries of $J$ -- that is, the set of all complex-analytic isometries of $\Cx$ preserving $J$. A rational function is exceptional if it is conformally conjugate to a power map, a Chebyshev polynomial or a Lattès map. This result highlights the importance of the group of symmetries for the Julia set of polynomials; it completely determines which polynomials share that Julia set.

The generalisation to rational maps, however, is not completely understood yet. Levin and Przytycki proved in 1997 that -- for a large class of rational functions -- having the same Julia set is equivalent to having the same measure of maximal entropy \cite{LP97}, while Ye proved that the characterisation given by (\ref{eq:Stein}) is not possible even for non-exceptional rational maps \cite{Ye15}. Here, we prove some partial extensions to Beardon's results on the symmetry group of Julia sets. We apply these results to obtain a complete description of the symmetries for maps of the form $z\mapsto z^m + \lambda/z^d$, previously studied by McMullen, Devaney and others \cite{McMullen88,DLU05}.

\section{Results}
Since rational functions are not holomorphic throughout all of $\Cx$, it is natural to consider them as analytic endomorphisms of the Riemann sphere $\Chat = \Cx\cup\{\infty\}$. Therefore, as opposed to Beardon's $\Isom(\Cx) = \{z\mapsto az + b : |a| = 1\}$, our set of possible symmetries shall be the set of holomorphic isometries of $\Chat$
$$ \Isom(\Chat) = \left\{z\mapsto \frac{az - \bar{b}}{bz + \bar{a}} : |a|^2 + |b|^2 = 1\right\}. $$
This set is isomorphic as a Lie group to $\SO3$ -- which means that there is a diffeomorphism $\Phi:\Isom(\Chat)\to\SO3$ that respects the group operations --, and as such it is compact and connected, but not simply connected. Given a rational function $R$, this allows us to put our first restriction on the structure of
$$ \Sigma(R) = \left\{\sigma\in\Isom(\Chat) : \sigma[J(R)] = J(R)\right\}, $$
the symmetries of its Julia set.

\begin{lemma} \label{lem:closed}
$\Sigma(R)$ is a closed set.
\end{lemma}
\begin{proof}
If $J(R) = \Chat$, then every isometry of the Riemann sphere preserves $J(R)$. Hence, $\Sigma(R)$ is $\Isom(\Chat)$, which is a closed group, and the conclusion follows.

If $J(R) \neq \Chat$, we cannot have $\Sigma(R) = \Isom(\Chat)$ -- since $\Isom(\Chat)$ acts transitively on the sphere, it would always be possible to map a point in the Julia set to some point outside the Julia set! Thus, we take $\sigma\in\Isom(\Chat)\setminus\Sigma(R)$. We know that there exists some $z\in J(R)$ such that $\sigma(z) \in F(R) = \Chat\setminus J(R)$, which is an open set. Therefore, there must be a neighbourhood $U$ of $\sigma(z)$ that does not intersect $J(R)$. Since $\Chat$ is a homogeneous space, this yields a neighbourhood $V\subset\Isom(\Chat)$ of the identity such that $\mu\sigma\notin\Sigma(R)$ for every $\mu\in V$. The construction of this neighbourhood $V$ is as follows: for every $w\in U$, the homogeneity of $\Chat$ implies the existence of some $\gamma\in\Isom(\Chat)$ such that $\gamma[\sigma(z)] = w$. Since the action of $\Isom(\Chat)$ on $\Chat$ is smooth, the collection of such $\gamma$ for every $w\in U$ yields a neighbourhood of the identity -- which is also in this collection for $w = \sigma(z)$. By continuity of the group operations, $V\sigma$ is a neighbourhood of $\sigma$ which does not intersect $\Sigma(R)$, and thus $\Isom(\Chat)\setminus\Sigma(R)$ is open.
\end{proof}

Though simple, this results has crucial consequences. Firstly, as a closed subset of a compact set, we get that $\Sigma(R)$ is compact; secondly, by Cartan's closed subgroup theorem, it follows that $\Sigma(R)$ is a Lie subgroup of $\Isom(\Chat)$ -- which means that it is an embedded submanifold of $\Isom(\Chat)$. Hence, we obtain our first serious restriction on $\Sigma(R)$.

\begin{theorem} \label{thm:struct}
For a rational map $R$, $\Sigma(R)$ is (isomorphic to) one of:
\begin{enumerate}[(i)]
    \item The trivial group;
    \item A group of roots of unity;
    \item A dihedral group generated by a root of unity $z\mapsto e^{2\pi i/k}z$ and an inversion $z\mapsto 1/z$;
    \item The orientation-preserving symmetries of a regular tetrahedron, octahedron or icosahedron;
    \item $\text{S}[\text{O}(1)\times\text{O}(2)]$ -- i.e., the group of isometries of the form $z\mapsto e^{i\theta}z$ and $z\mapsto e^{i\theta}/z$ for any $\theta\in[0, 2\pi)$;
    \item All isometries of the Riemann sphere.
\end{enumerate}
\end{theorem}

\begin{proof}
We note that there is little to be done in cases (i) and (vi) from a symmetry point of view. Although their dynamics may be interesting -- case (vi), for instance, are the Latt\`es maps -- , we assume now that $\Sigma(R)$ is neither trivial nor all of $\Isom(\Chat)$.

The first distinction we must make is between a discrete and a continuous symmetry group. In the former, $\Sigma(R)$ must be a discrete Lie group, and -- since $\Isom(\Chat)$ is compact -- this implies that it is finite. The classification of finite subgroups of $\SO3$ in \cite{Carne}, Theorem 4.1, then gives us cases (ii) through (iv). We do remark that we have changed the nomenclature; Carne refers to the roots of unity as the symmetries of a cone on a regular plane polygon, and to the dihedral groups as symmetries of a double cone on a regular plane polygon.

For continuous symmetry groups, we must study the Lie subgroups of $\Isom(\Chat)$. Take, then, the connected component $H$ of $\Sigma(R)$ containing the identity -- which must be a Lie subgroup of $\Isom(\Chat)$ with an associated Lie subalgebra $\mathfrak{h}\subset\mathfrak{i}(\Chat)$. Since $\mathfrak{i}(\Chat)\simeq\mathfrak{so}(3)\simeq\mathbb{R}^3$, where the Lie algebra structure is given by the vector product, it follows that the only proper non-trivial Lie subalgebras of $\mathfrak{i}(\Chat)$ are one-dimensional, and thus $H$ is a one-dimensional Lie subgroup of $\Isom(\Chat)$.

Now, every one-dimensional Lie group admits a parametrisation using the Lie exponential. Therefore, we write $H = \{\exp[tX] : t\in\mathbb{R}\}$ for some $X$ in its Lie algebra $\mathfrak{h}\subset\mathfrak{i}(\Chat)$. The action of $\Isom(\Chat)$ on the Riemann sphere yields a flow $\phi:\mathbb{R}\times\Chat\to\Chat$ defined as $\phi(t, z) = \exp[tX](z)$, which has as associated vector field $\vec{X}:T\Chat\to T\Chat$ --  called the action field of $H$ -- given by
\begin{equation*}
    \vec{X}(z) = \partial_t\phi(0, z) = \left.\frac{d}{dt}\exp[tX]z\right|_{t=0}.
\end{equation*}
By the hairy ball theorem, there exists $z_0\in\Chat$ such that $\vec{X}(z_0) = 0$; this point satisfies $\phi(t, z_0) = z_0$ for all $t$, and so it is fixed by every $\sigma\in H$. As every isometry of $\Chat$ has exactly two antipodal fixed points, it follows that $z_0$'s antipode is also fixed by every element of $H$. Conjugating the Riemann sphere by an isometry such that $z_0 = 0$, we conclude that $H$ is conjugate to $S^1 = \{e^{it} : t\in[0, 2\pi)\}$.

Now, take a $z\in J(R)$ that is not fixed by the action of $H$. Its orbit must be a circle, and so $J(R)$ is a collection of circles -- either a single circle or an uncountable amount of them. In the former case, $\Sigma(R)$ are the symmetries of a circle, so its elements are either of the form $z\mapsto e^{i\theta}z$ or $z\mapsto e^{i\theta}/z$ (furthermore, Eremenko and van Strien \cite{EvS11} showed that either $R$ or $R^2$ must be a Blaschke product). We show that the latter case is not possible. The smoothness of each connected component of $J(R)$ implies that all multipliers of repelling periodic orbits are real \cite{EvS11,Milnor} and, still following Emerenko and van Strien, this implies that $J(R)$ is contained in a single circle -- and therefore is a single circle. We fall back to the previous case, and we are done.
\end{proof}

Now, we need conditions that allow us to choose, among all these possible geometries, which one corresponds to a given rational map $R$. We offer a sufficient condition and a necessary one; Ye's results suggest that complete characterisations, like for polynomials, are not possible.

\begin{proposition} \label{prop:suff}
Let $R$ be a rational function and $\sigma\in\Isom(\Chat)$, and suppose that $R$ is not a Latt\`es map. If $R\sigma = \sigma^kR$ for some $k\geq 1$, then $\sigma\in\Sigma(R)$.
\end{proposition}
\begin{proof}
We will show that the Fatou set of $R$ is invariant under $\sigma$. Take $z\in F(R)$. By the Arzel\`a-Ascoli theorem, for any $\epsilon > 0$ there is a neighbourhood $U$ of $z$ satisfying $\text{diam}\left[R^m(U)\right] < \epsilon$ for every $m\geq 1$. Now, we consider how $R$ behaves at $\sigma(z)$. By induction, our hypothesis implies that there exists for all $m\geq 1$ some $l\geq 1$ (which depends on $m$) such that $R^m\sigma = \sigma^lR^m$. Indeed, for the case $m = 1$, $l$ is easily given by $k$ as per our hypothesis. Now, for any $m$, $R^{m+1}\sigma = R(R^m\sigma)$, and the induction hypothesis gives us $R^{m+1}\sigma = R(\sigma^lR^m)$. By using that $R\sigma = \sigma^kR$, we can shift the $\sigma$'s ``one-by-one'' to obtain $R^{m+1}\sigma = \sigma^{kl}R^{m+1}$. Therefore, $\text{diam}\left[R^m\sigma(U)\right] = \text{diam}\left[\sigma^lR^m(U)\right]$; since $\sigma$ is an isometry of the Riemann sphere, it leaves the diameter of a set unchanged, and thus $\text{diam}\left[R^m\sigma(U)\right] = \text{diam}\left[R^m(U)\right]$ for every $m$. Since the terms on the right-hand side are limited by $\epsilon$, this implies (by the Arzel\`a-Ascoli theorem) that $R^m$ is a normal family at $\sigma(z)$, and thus $\sigma[F(R)] \subset F(R)$. Since $\sigma^{-1}$ is also an isometry, we can apply the same reasoning to conclude that $\sigma^{-1}[F(R)] \subset F(R)$, and so $F(R)$ -- and thus $J(R)$ -- is invariant under $\sigma$ and $\sigma\in\Sigma(R)$.
\end{proof}

Our necessary condition even allows us to specify a value for $k$ in Proposition \ref{prop:suff}, albeit in a very specific situation. However, we shall need some technical results concerning potentials in order to prove it. We refer to \cite{Ransford95}, \cite{Klimek} and \cite{Berteloot} for the necessary concepts.

\begin{lemma} \label{lem:dec}
Let $\Omega$ be a domain on the Riemann sphere and $z_1, z_2, \ldots, z_k$ be points in $\Omega$. Suppose $f:\Omega\to(0, \infty]$ is a function such that:
\begin{enumerate}[(i)]
    \item $f$ is harmonic on $\Omega\setminus\{z_1, \ldots, z_k\}$;
    \item As $z\to z_i$, there exists some $m_i > 0$ such that$f(z) = -m_i\log|z - z_i| + O(1)$ for every $i = 1, \ldots, k$ (we say that $f$ has a logarithmic pole of order $m_i$ at $z_i$);
    \item As $z\to\partial\Omega$, $f(z) \to 0$.
\end{enumerate}
Then, $f$ can be decomposed as
$$ f(z) = \sum_{i=1}^k m_ig_\Omega(z, z_i), $$
where $g_\Omega(z, w)$ denotes the Green's function for $\Omega$ with pole at $w$.
\end{lemma}
\begin{proof}
Firstly, we notice that if $k = 1$ then the requirements for $f$ are exactly those for a Green's function of a domain, and the uniqueness of the Green's function gives us the desired result (albeit trivially). For $k > 1$, consider the function
$$ g(z) = \frac{1}{m_1}\left[f(z) - \sum_{i=2}^k m_ig_\Omega(z, z_i)\right]. $$
It approaches zero as $z\to\partial\Omega$; it has a simple logarithmic pole at $z_1$; and it is harmonic on $\Omega\setminus\{z_1, \ldots, z_k\}$. For the points $z_2, \ldots, z_k$, its limit as $z\to z_i$ is finite and so, by the removable singularity theorem, it can be extended as a harmonic function to $\Omega\setminus\{z_1\}$. Denoting this extension by $g$ in an abuse of notation, we see that $g$ satisfies the requirements for the Green's function of $\Omega$ with pole at $z_1$, and so
$$ f(z) - \sum_{i=2}^k m_ig_\Omega(z, z_i) = m_1g_\Omega(z, z_1). $$
\end{proof}

Next, we would like to prove that symmetries of $J(R)$ somehow respect the critical points of $R$, which are central to its dynamical behaviour. Indeed, let $C(R)$ stand for the set of critical points of $R$ together with their pre-images (we shall call it the pre-critical set):
$$ C(R) = \bigcup_{n\geq 0} R^{-n}\{z\in\Chat : \text{$z$ is a critical point of $R$}\}. $$
We shall demonstrate that any symmetry of $J(R)$ must, if $R$ satisfies adequate conditions, preserve $C(R)$.

First, it is well known that any rational function $R$ admits a unique measure of maximal entropy $\mu_R$ (this was proved independent by Lyubich \cite{Lyu83} and Freire, Lopes and Ma\~{n}\'e \cite{FLM83} in 1983). Then, Levin and Przytycki showed that, for non-exceptional rational functions -- i.e., functions whose Julia set is not smooth and whose Fatou set does not contain any parabolic domains, Siegel disks or Herman rings --, having the same measure of maximal entropy is in fact equivalent to having the same Julia set. We are going to associate a potential to this fundamental measure of rational maps, denoted its ergodic potential, and show that it is continuous and invariant under symmetries of the Julia set. Then, we shall show that its local minima coincide with the points in $C(R)$, allowing us to conclude that symmetries must preserve it.

\begin{definition}
Let $R$ be a rational function, and $\mu_R$ its unique measure of maximal entropy. We define its ergodic potential to be the function $u_R:\Chat\to\mathbb{R}$ given by
$$ u_R(z) = \int_{\Chat} \log\frac{1}{\rho(z, w)}\,d\mu_R(w), $$
where $\rho$ denotes the chordal metric on the Riemann sphere.
\end{definition}

In potential theory, this would be known as the elliptic potential associated to the measure $\mu_R$. Since the logarithm is a subharmonic function, $u_R$ is subharmonic in $\Chat\setminus J(R)$, and Okuyama \cite{Oku05} proved that it is actually continuous. He also proved that it satisfies $dd^cu_R = \omega - \mu_R$, where $\omega$ is the standard area form on the Riemann sphere and $d$ and $d^c$ are the differential operators given by $d = \partial + \overline{\partial}$ and $d^c = (i/(2\pi))(\overline{\partial} - \partial)$.

\begin{proposition} \label{prop:erg}
For a non-exceptional rational map $R$ with $\sigma\in\Sigma(R)$, $\mu_R$ and $u_R$ are invariant under $\sigma$.
\end{proposition}
\begin{proof}
As discussed in Lemma \ref{lem:iff}, $\sigma\in\Sigma(R)$ implies that $J(\sigma R) = J(\sigma)$. This means that $\mu_{\sigma R} = \mu_R$, and thus, since the measure of maximal entropy is invariant, we get
$$ \mu_R = (\sigma R)_*\mu_R = \sigma_*R_*\mu_R = \sigma_*\mu_R(w). $$
In other words, $\mu_R$ is invariant under symmetries of the Julia set. Now, the expression for $u_R\circ\sigma$ becomes
$$ u_R\circ\sigma(z) = \int_{\Chat} \log\frac{1}{\rho[\sigma(z), w]}\,d\mu_R(w) $$
and, as $\sigma$ is an isometry of the metric $\rho$,
$$ u_R(z) = \int_{\Chat} \log\frac{1}{\rho[z, \sigma^{-1}(w)]}\,d\mu_R(w) = \int_{\Chat} \log\frac{1}{\rho(z, w)}\,d(\sigma_*\mu_R)(w). $$
Since $\mu_R$ is invariant under $\sigma$, we recover the original expression for $u_R$ from the right-hand side of the equality above, and thus $(\sigma^*u_R)(z) := u_R\circ\sigma(z) = u_R(z)$.
\end{proof}

Now, since $u_R$ is continuous and $\Chat$ is compact, it attains maxima and minima. Since it is subharmonic outside of the Julia set, we conclude that its maxima are attained on $J(R)$ and its minima, in $F(R)$. The invariance  of $u_R$ under symmetries implies, in particular, that local minima of $u_R$ are mapped into other local minima; upon proving that the local minima of $u_R$ coincide with $C(R)$, we will have proved the following.

\begin{theorem} \label{thm:Cr}
For any non-exceptional rational map $R$ and $\sigma\in\Sigma(R)$, $\sigma$ preserves the pre-critical set.
\end{theorem}
\begin{proof}
As discussed above, what remains is to prove that the minima of $u_R$ are found in $C(R)$. For that, we shall appeal to the polynomial lift of $R$ in $\Cx^2$. Let $R$ be written in the form $R(z) = P(z)/Q(z)$, where $P$ and $Q$ are co-prime polynomials. If $d\geq 2$ is the degree of $R$, then the function
$$ \widehat{R}(z_1, z_2) = z_2^d\left(P(z_1z_2^{-1}), Q(z_1z_2^{-1})\right) $$
is a homogeneous polynomial of degree $d$ in $\Cx^2$. Furthermore, if $\pi:\Cx^2\to\Cx$ is the standard projection from $\Cx^2$ onto the Riemann sphere (given by $\pi(z_1, z_2) = z_1/z_2$), it is readily seen that $R\circ\pi = \pi\circ\widehat{R}$. Take, then, the escape rate function of $\widehat{R}$, defined as
$$ G_R(z_1, z_2) = \lim_{n\to\infty} \frac{1}{d^n}\log\|\widehat{R}^n(z_1, z_2)\|. $$
The definition implies that $G_R$ is plurisubharmonic. It is also known \cite{Berteloot} that there exists a unique function $g:\Chat\to\mathbb{R}$ such that
$$ G_R - \log\|\cdot\| = g\circ\pi. $$
Furthermore, this function $g$ satisfies $dd^cg + \omega = \mu_R$. Since $\mu_R = \omega - dd^cu_R$ as well, we conclude that $dd^c(g + u_R) = 0$ -- or, in other words, $g + u_R$ is harmonic throughout $\Chat$. Since the only harmonic functions on the Riemann sphere are constant, it follows that $u_R = -g + C$ for some real constant $C$, which we shall promptly ignore.

Thus, we have that $\log\|\cdot\| - G_R = \pi^*u_R$. By Proposition \ref{prop:erg}, $\sigma^*u_R = u_R$, so that $\log\|\cdot\| - G_R = \pi^*(\sigma^*u_R) = (\sigma\pi)^*u_R$. $\sigma$, being an isometry of $\Chat$, lifts to an isometry $\Sigma$ of $\Cx^2$, and therefore
$$ \log\|\cdot\| - G_R = \Sigma^*(\pi^*u_R) = \Sigma^*(\log\|\cdot\| - G_R) = \Sigma^*\log\|\cdot\| - \Sigma^*G_R. $$
The fact that $\Sigma$ is an isometry implies that the logarithm cancels out, so that $G_R = \Sigma^*G_R$. In particular, local extrema of $G_R$ over $F(R)$ are mapped by $\Sigma$ onto local extrema of $G_R$ over $F(R)$; the definition of $G_R$ and the chain rule imply that these are exactly the points with $J\widehat{R} = 0$ (where $J\widehat{R}$ denotes the Jacobian of $\widehat{R}$) together with their pre-images. Since points with $J\widehat{R} = 0$ correspond to lifts of critical points of $R$, we are done.
\end{proof}

We are now ready to prove the necessary condition.

\begin{proposition} \label{prop:nec}
Suppose $R$ is non-exceptional and $\sigma\in\Sigma(R)$ fixes a superattracting fixed point $z_0$ of $R$, with local degree $m > 1$. Then, $R\sigma = \sigma^mR$.
\end{proposition}
\begin{proof}
Consider the function
$$ f(z) = -\lim_{n\to\infty} \frac{1}{m^n} \log |\Phi[R^m(z)]|, $$
where $\Phi:U\to B(0; r)$ is a biholomorphism conjugating $R$ to $z\mapsto z^m$. It is well-defined throughout the immediate basin of attraction for $z_0$, denoted $\mathcal{A}(z_0)$, with poles at pre-images of $z_0$ with order given by the multiplicity of the pre-image. By B\"ottcher's theorem, $R$ sends level curves of $f$ onto level curves of $f$ -- in fact, $f[R(z)] = mf(z)$. Furthermore, we can apply Lemma \ref{lem:dec} to $f$ and write
$$ f(z) = \sum_{i=0}^k m_ig_{F_0}(z, z_i), $$
where $F_0$ is the connected component of $F(R)$ containing $z_0$, and we have enumerated the pre-images of $z_0$  in $F_0$ as $z_0, z_1, \ldots, z_k$ ($z_0$ is a pre-image of itself). Now, we have:
$$ f[\sigma(z)] = \sum_{i=0}^k m_ig_{F_0}[\sigma(z), z_i] = \sum_{j=0}^k m_jg_{F_0}[\sigma(z), \sigma(z_j)], $$
where in the last inequality we have used Theorem \ref{thm:Cr} to ensure that $\sigma$ permutes the $z_i$, and permuted the indices accordingly. Next, our hypothesis that $\sigma$ is a symmetry of $J(R)$ fixing $z_0$ implies that $\sigma$ is a conformal mapping of $F_0$ onto itself, and -- since Green's functions are preserved by conformal mappings -- we conclude that $g_{F_0}[\sigma(z), \sigma(z_i)] = g_{F_0}(z, z_i)$ for every $i$. Therefore, $f\sigma(z) = f(z)$, which means that there is a neighbourhood $V$ of $z_0$ that is forward-invariant under both $R$ and $\sigma$ and which is contained in $U$ -- one need only define $V$ as any level curve of $f$ that is completely contained in $U$. Again by B\"ottcher's theorem, $\Phi$ maps $V$ into a circle $B(0; \delta)$ and conjugates $R$ to $z^m$.

Now, consider the functions $\widehat{\sigma} = \Phi\sigma\Phi^{-1}:B(0; \delta)\to B(0; \delta)$ and $\widehat{R} = \Phi R\Phi^{-1}:B(0; \delta)\to B(0; \delta)$. We already know that $\widehat{R}(z) = z^m$; now, notice that $\widehat{\sigma}$ is an automorphism of $B(0, \delta)$, and thus it is an isometry of the hyperbolic metric on $B(0; \delta)$. Since it also fixes $0$, it follows that it is of the form $\widehat{\sigma}(z) = e^{i\theta}z$. Therefore,
$$ \widehat{R}\widehat{\sigma}(z) = (e^{i\theta}z)^m = e^{i(m\theta)}z^m = \widehat{\sigma}^m\widehat{R}(z), $$
and so $R\sigma = \sigma^mR$.
\end{proof}

In Figure \ref{fig:syms}, we show examples of Julia sets with finite, non-trivial symmetry groups. Figure \ref{subfig:a} is the Julia set for the Newton map of the polynomial $z\mapsto z^3 + 1$, and Figure \ref{subfig:b} corresponds to the map $z\mapsto z^2 + 1/z^2$. The latter belongs to a family which shall be discussed in further details in section \ref{sec:app}.

\begin{figure}[!h]
    \centering
    \begin{subfigure}[b]{0.4\textwidth}
        \includegraphics[width=\textwidth]{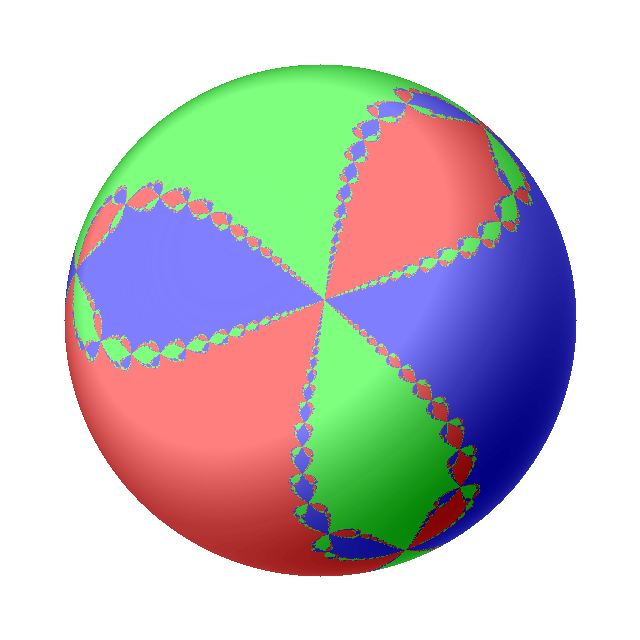}
        \caption{Phase portrait for the Newton map of $z\mapsto z^3 + 1$. Each colour represents the attraction basin of minus a cube root of unity.}
        \label{subfig:a}
    \end{subfigure}
    \quad
    \begin{subfigure}[b]{0.45\textwidth}
        \includegraphics[width=\textwidth]{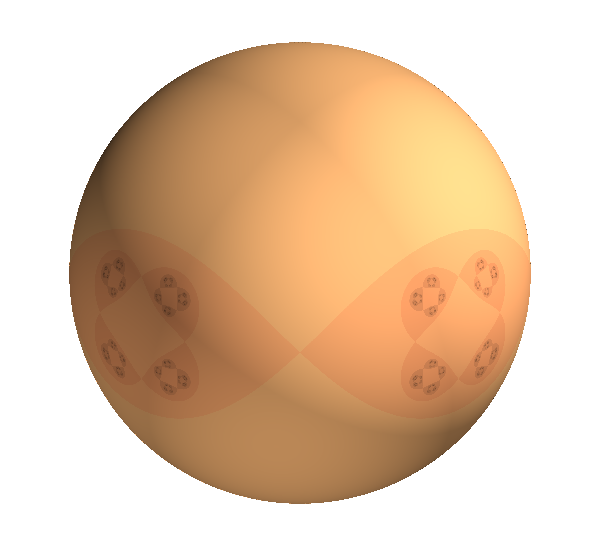}
        \caption{Phase portrait for the rational map $z\mapsto z^2 + 1/z^2$. Lighter points belong to the basin of attraction of infinity, and black points are in the Julia set.}
        \label{subfig:b}
    \end{subfigure}
    \caption{Examples for rational maps with finite symmetry groups.}
    \label{fig:syms}
\end{figure}

Finally, a known application of symmetries of a Julia set is in the description of all polynomials that share a Julia set \cite{SS95}. Although Ye's results provide an example where the simple criterion for polynomials fails for rational maps, we can nevertheless offer a sufficient condition involving symmetries. First, however, we shall need this technical lemma.

\begin{lemma} \label{lem:iff}
Let $R$ and $S$ be rational maps of degree $\geq 2$. Then, $J(R) = J(S) \Leftrightarrow J(R)$ is completely invariant under $S$ and $J(S)$ is completely invariant under $R$. In particular, if $\sigma\in\Sigma(R)$, then $J(R) = J(\sigma R)$.
\end{lemma}
\begin{proof}
If $J(R) = J(S)$, the complete invariance of the Julia set follows immediately from its definition. To prove the converse, we recall that $J(S)$ is characterised as the minimal closed set with more than three points which is completely invariant under $S$, hence $J(R)\subset J(S)$. By the symmetry of the hypothesis, we also get that $J(S)\subset J(R)$, and so they are equal.

Now, consider $\sigma\in\Sigma(R)$. In order to conclude that $J(R) = J(\sigma R)$, we shall prove that $J(R)$ is invariant under $\sigma R$, and vice-versa. Firstly, since $\sigma$ is a symmetry of $J(R)$, we have by definition that $\sigma[J(R)] = J(R)$ and so it is clear that $\sigma R[J(R)] = J(R)$. Also, by the minimality of the Julia set, this implies that $J(\sigma R) \subset J(R)$.

All that is left is to prove that $R[J(\sigma R)] = R^{-1}[J(\sigma R)] = J(\sigma R)$, and we shall do it by contradiction. Suppose, then, that $J(\sigma R)$ is not backward invariant under $R$ (since $R$ is surjective, this is actually equivalent to assuming that $J(\sigma R)$ is not completely invariant). Thus, we can take $z\in J(\sigma R)$ such that $R^{-1}(z)$ contains at least one point -- which, by an abuse of notation, we shall denote by $R^{-1}(z)$ -- that is not in $J(\sigma R)$. In other words, $R^{-1}(z)\in F(\sigma R)$. However, we already know that $J(\sigma R)$ is a subset of $J(R)$, which is completely invariant under $R$, and therefore $R^{-1}(z)\in J(R)$. This will be the basis for obtaining a contradiction.

Since $R^{-1}(z)$ is in the Fatou set of $\sigma R$, it follows from the Arzel\'a-Ascoli theorem that $\{(\sigma R)^k\}_{k\geq 1}$ is equicontinuous there. Thus, for any $\epsilon > 0$, we can take a neighbourhood $U$ of $z$ such that
$$ \text{diam}\left[(\sigma R)^kR^{-1}(U)\right] < \epsilon\quad \text{for every } k\geq 1. $$
By taking a term from the family of $(\sigma R)^k$, this becomes
$$ \text{diam}\left[(\sigma R)^{k-1}\sigma(U)\right] < \epsilon,\quad k\geq 1. $$
Next, we consider what the sequence of mappings $(\sigma R)^{k-1}\sigma$ means for the diameter. $\sigma$ is an isometry of the Riemann sphere; therefore, none of the $\sigma$ terms in this expression have any effect on diameter. This means that the end result of $\text{diam}\left[(\sigma R)^{k-1}\sigma(U)\right]$ is ultimately determined by the iterations of $R$. Also, since $z\in J(R)$ and $J(R)$ is completely invariant under both $\sigma$ and $R$, this means that $U$ is always mapped to a neighbourhood of a point in $J(R)$. Since $R$ eventually expands all neighbourhoods of points in its Julia set, we can conclude that $\text{diam}\left[(\sigma R)^{k-1}\sigma(U)\right]$ should eventually grow larger than any small value of $\epsilon$, and so we have reached a contradiction.
\end{proof}

We are now in a position to prove the following.

\begin{proposition} \label{prop:sufJ}
Let $R$ and $S$ be rational maps of degree $\geq 2$ such that $SR = \sigma RS$ for some $\sigma\in\Sigma(R)$. Then, $R$ and $S$ have the same Julia set.
\end{proposition}
\begin{proof}
We shall prove that, under the hypotheses of the theorem, $F(R)$ is completely invariant under $S$ and vice-versa. Since both are surjective on $\Chat$, it suffices to prove backward invariance, i.e., $S^{-1}[F(R)\subset F(R)\subset S^{-1}[F(R)]$.

Firstly, notice that, for all $k\geq 1$, we obtain by induction -- the argument is analogous to the one used in Proposition \ref{prop:suff} -- that $SR^k = (\sigma R)^kS$. Now, let $M$ be a Lipschitz constant for $S$ in the spherical metric. For $z\in F(R)$, the definition of the Fatou set means that $\{R^k\}_{k\geq 1}$ is normal, and therefore equicontinuous by the Arzel\'a-Ascoli theorem, at $z$. As such, for any $\epsilon$, there exists a neighbourhood $U$ of $z$ such that $\mathrm{diam}\left[R^k(U)\right] <\epsilon/M$ for every $k\geq 1$. Since $\sigma$ is an isometry of the Riemann sphere, we have:
\begin{equation*}
    \mathrm{diam}\left[(\sigma R)^kS(U)\right] = \mathrm{diam}\left[SR^k(U)\right] \leq M\mathrm{diam}\left[R^k(U)\right] < \epsilon.
\end{equation*}
This tells us that $\{(\sigma R)^k\}_{k\geq 1}$ is equicontinuous on $S(U)$, and thus $S(z)\in F(\sigma R) = F(R)$. Therefore, $S[F(R)]\subset F(R)$, and $F(R)\subset S^{-1}[F(R)]$.

Now, let $V = S^{-1}$ for $U\subset F(R)$. Since $F(R) = F(\sigma R)$, we can pick $U$ such that $\mathrm{diam}[(\sigma R)^k(U)] < \epsilon$ for every $k\geq 1$ for an arbitrary choice of $\epsilon > 0$. Then,
\begin{equation*}
\begin{split}
    \mathrm{diam}\left[SR^k(V)\right] =& \mathrm{diam}\left\{SR^k\left[S^{-1}(U)\right]\right\}\\
    =& \mathrm{diam}\left\{(\sigma R)^kS\left[S^{-1}(U)\right]\right\}\\
    =& \mathrm{diam}\left[(\sigma R)^k(U)\right] < \epsilon,
\end{split}
\end{equation*}
and so $\{SR^k(V)\}_{k\geq 1}$ is equicontinuous on $V$. Since $S$ is Lipschitz continuous, so is $\{R^k\}_{k\geq 1}$, and by the Arzel\'a-Ascoli theorem we have that $\{R^k\}_{k\geq 1}$ is normal on $S^{-1}[F(R)]$. It follows that $S^{-1}[F(R)]\subset F(R)$, and so we can conclude that $F(R)$ is backward invariant under $S$. Finally, this implies that $F(R)$ -- and thus $J(R)$ -- is completely invariant under $S$.

For the remaining statement -- the complete invariance of $F(S)$ under $R$ --, we recall that $\Sigma(R)$ is a group; hence, $\sigma^{-1}$ is also a symmetry of the Julia set, and it satisfies $RS = \sigma^{-1}SR$. We can apply the same argument as above, concluding that $F(S)$ is completely invariant under $R$, and by Lemma $\ref{lem:iff}$ we obtain $J(R) = J(S)$.
\end{proof}

\section{Applications} \label{sec:app}
We apply these results to the families of singularly perturbed maps, also called McMullen maps. These are rational functions of the form
\begin{equation*}
    R_\lambda(z) = z^m + \frac{\lambda}{z^d}
\end{equation*}
for $m\geq 2$, $d\geq 1$ and $\lambda\in\Cx$. They have been previously studied by McMullen \cite{McMullen88}, Devaney and others \cite{DLU05}, who have already exploited particular M\"obius functions that preserve $J(R_\lambda)$. We provide a way to determine all isometries of $\Chat$ that do so.

\begin{theorem}
The Julia set of $R_\lambda$ has the following symmetries:
\begin{enumerate}[(i)]
    \item $z\mapsto e^{i\theta}z^{\pm 1}$ for any $\theta\in\mathbb{R}$, if $\lambda = 0$;
    \item $z\mapsto\mu z^{\pm 1}$, where $\mu^{m+d} = 1$, if $m = d$ and $|\lambda| = 1$;
    \item $z\mapsto\mu z$, where $\mu^{m+d} = 1$, otherwise.
\end{enumerate}
\end{theorem}
\begin{proof}
The case $\lambda = 0$ reduces to $R_0(z) = z^m$, and its Julia set is a circle; hence, (i) follows from Theorem \ref{thm:struct}.

Now, if $\lambda \neq 0$, any symmetry must either fix infinity or map it to another point. We start with the symmetries fixing infinity. These are, of course, a subgroup of $\Sigma(R_\lambda)$ made of symmetries of the form $z\mapsto e^{i\theta}z$ for some values of $\theta\in\mathbb{R}$. Our task here is to ascertain the possible values of $\theta$.
Let $\sigma(z) = \mu z$, where $\mu^{m+d} = 1$. We shall prove that $\sigma\in\Sigma(R_\lambda)$. We have:
\begin{equation*}
    R_\lambda\sigma(z) = \mu^mz^m + \frac{\lambda}{\mu^dz^d} = \mu^mz^m + \mu^m\frac{\lambda}{z^d} = \mu^mR_\lambda(z) = \sigma^mR_\lambda(z),
\end{equation*}
and so, by Proposition \ref{prop:suff}, $\sigma$ is a symmetry of $J(R_\lambda)$. On the other hand, any symmetry fixing infinity must, by Proposition $\ref{prop:nec}$, satisfy $R_\lambda\sigma = \sigma^mR_\lambda$. If $\sigma(z) = \nu z$ with $|\nu| = 1$, then
\begin{equation*}
    R_\lambda\sigma(z) = \nu^mz^m + \nu^{-d}\frac{\lambda}{z^d} = \nu^mz^m + \nu^m\frac{\lambda}{z^d} = \sigma^mR_\lambda(z),
\end{equation*}
and thus $\nu^{ -d} = \nu^m$ and $\nu^{m+d} = 1$. This means that the symmetries fixing infinity are a subgroup of $\Sigma(R_\lambda)$ isomorphic to $\mathbb{Z}/(m+d)\mathbb{Z}$.

Now, we consider any remaining symmetries. First, we invoke Theorem \ref{thm:struct}: since $\Sigma(R_\lambda)$ only admits certain structures, any symmetry group that properly contains $\{z\mapsto \nu z : \nu^{m+d} = 1\}$ as a subgroup must also contain a symmetry of the form $\sigma(z) = \mu/z$ with $\mu^{m+d} = 1$. In other words, the only possible structures for $\Sigma(R_\lambda)$ are $\{z\mapsto \mu z : \mu^{m+d} = 1\}$ or $\{z\mapsto \mu z : \mu^{m+d}=1\}\cup\{z\mapsto \mu/z : \mu^{m+d}=1\}$. What is left, thus, is to decide when it is one or the other.

Thus, consider that $\sigma$ has the form $\sigma(z) = \mu/z$. We shall again call upon results from potential theory. Let $g_0(z, 0)$ and $g_\infty(z, \infty)$ denote the Green's functions for the connected components of $F(R_\lambda)$ containing $0$ and $\infty$, respectively. Since conformal mappings send Green's functions to Green's functions, we have that $g_\infty[\sigma(z), \infty] = g_0(z, 0)$ for $z$ in a neighbourhood of $0$. At the same time, $R_\lambda$ maps this neighbourhood of $0$ to a neighbourhood of infinity with multiplicity $d$, and thus $g_\infty[R_\lambda(z), \infty] = dg_0(z, 0)$ by the uniqueness of the Green's function of a domain. Therefore, $g_\infty[R_\lambda(z), \infty] = dg_\infty[\sigma(z), \infty]$ and $d\log|\Phi\sigma(z)| = \log|\Phi R_\lambda(z)|$, where $\Phi$ is the B\"ottcher function for $R_\lambda$; hence, by applying the series expansion for $\Phi$,
\begin{equation*}
    \left[\frac{\mu}{z} + a_0 + \cdots\right]^d = \alpha\left[z^m + \frac{\lambda}{z^d} + \cdots\right],
\end{equation*}
where $|\alpha| = 1$. Comparing the coefficients in the series expansion, we conclude that $m = d$ and, simultaneously, $\mu^d = \alpha\lambda$. Since $|\mu| = |\alpha| = 1$, it follows that $|\lambda| = 1$ and we are done.
\end{proof}

Figures \ref{fig:ppl22} and \ref{fig:ppl21} illustrate this theorem. If $m = d = 2$, the isometries $z\mapsto \mu z$, where $\mu^4 = 1$, are always a symmetry of $J(R_\lambda)$ (see Figures \ref{fig:js22l01} and \ref{fig:js22l10}). If, in addition, $\lambda = 0$ (red dot) or $|\lambda| = 1$ (blue circle), then additional symmetries arise: in the former case, $J(R_0)$ is the unit circle, and so has all rotations as its symmetries as well as inversions with respect to the unit circle. In the latter, composing a rotation by a fourth root of unity with an inversion also yields a symmetry of $J(R_\lambda)$ (see Figure \ref{fig:js22l1}). If, on the other hand, $m = 2$ but $d = 1$, the region $|\lambda| = 1$ has nothing special with regards to symmetry. For any $\lambda\in\Cx^*$, the symmetry group consists of rotations by $\pi/3$ and $2\pi/3$ radians as in Figures \ref{fig:js21l01}, \ref{fig:js21l1} and \ref{fig:js21l10}. For $\lambda = 0$ (red dot) the Julia set is again the unit circle.

The figures show the connectedness locus -- i.e., the values of $\lambda$ for which $J(R_\lambda)$ is connected -- in order to emphasise one thing: the structure of the symmetry group has no regards for any topological changes to $J(R_\lambda)$. Indeed, while the structure of the Julia set changes drastically on the boundaries of the black region, the symmetry group ``ignores'' these changes and instead undergoes change as $|\lambda| = 1$ for $m = d = 2$, as in Figure \ref{fig:ppl22}.

\begin{figure}[H]
    \centering
        \begin{subfigure}[!t]{\textwidth}
        \centering
        \includegraphics[width=0.9\textwidth]{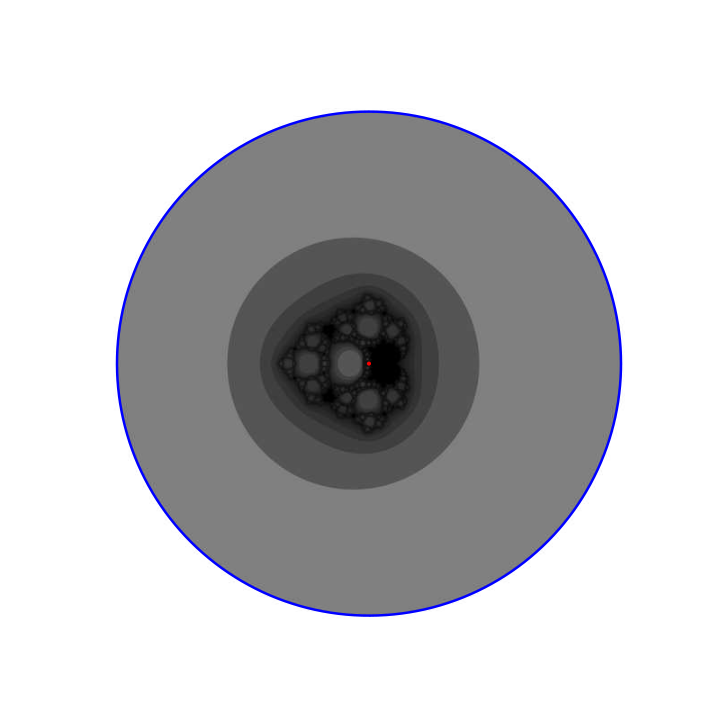}
        \caption{Parameter plane for the family $z\mapsto z^2 + \lambda/z^2$. We highlight the regions where $\Sigma(R_\lambda)$ is different.}
        \label{fig:ppl22}
    \end{subfigure}

    \begin{subfigure}[!b]{0.45\textwidth}
        \centering
        \includegraphics[width=0.45\textwidth]{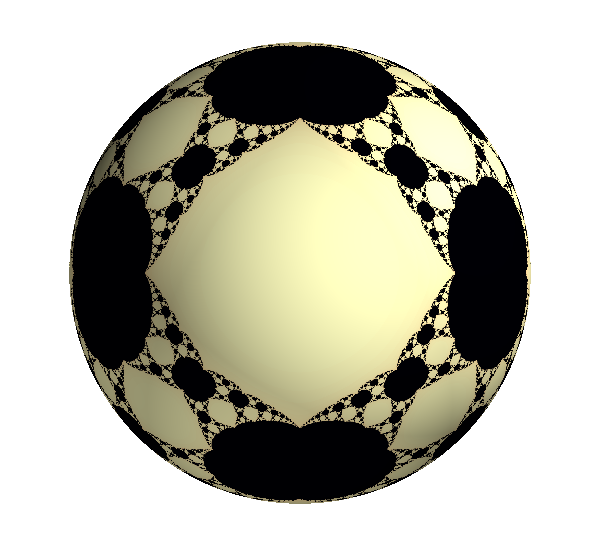}
        \caption{$\lambda = 0.1$}
        \label{fig:js22l01}
    \end{subfigure}
    \begin{subfigure}[!b]{0.45\textwidth}
        \centering
        \includegraphics[width=0.45\textwidth]{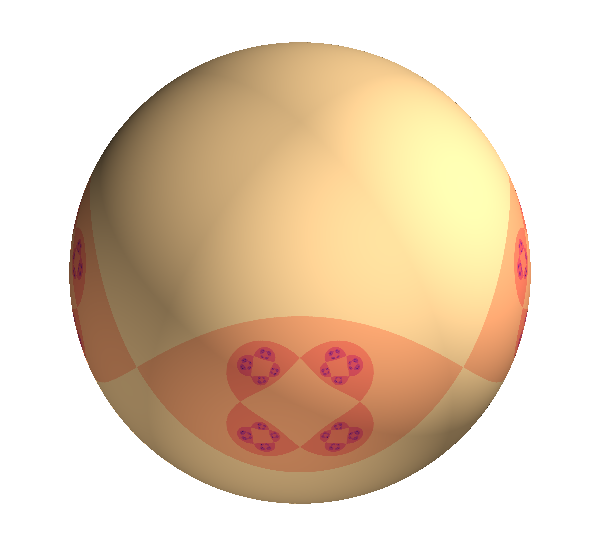}
        \caption{$\lambda = 1$}
        \label{fig:js22l1}
    \end{subfigure}
    \begin{subfigure}[!b]{0.45\textwidth}
        \centering
        \includegraphics[width=0.45\textwidth]{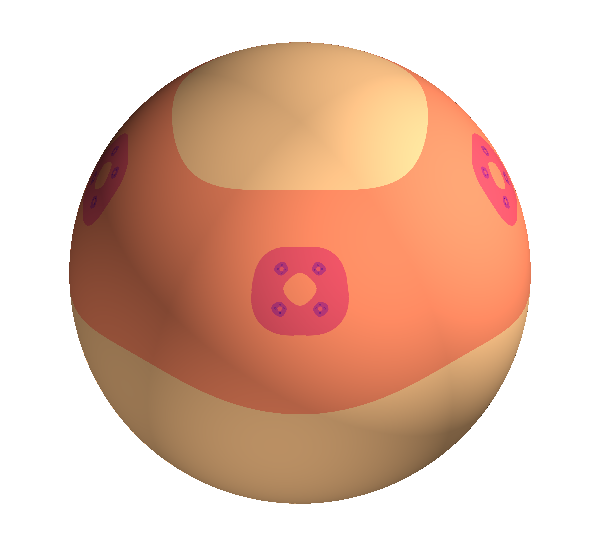}
        \caption{$\lambda = 10$}
        \label{fig:js22l10}
    \end{subfigure}
\end{figure}

\begin{figure}[H]
    \centering
    \begin{subfigure}[!t]{\textwidth}
        \centering
        \includegraphics[width=0.9\textwidth]{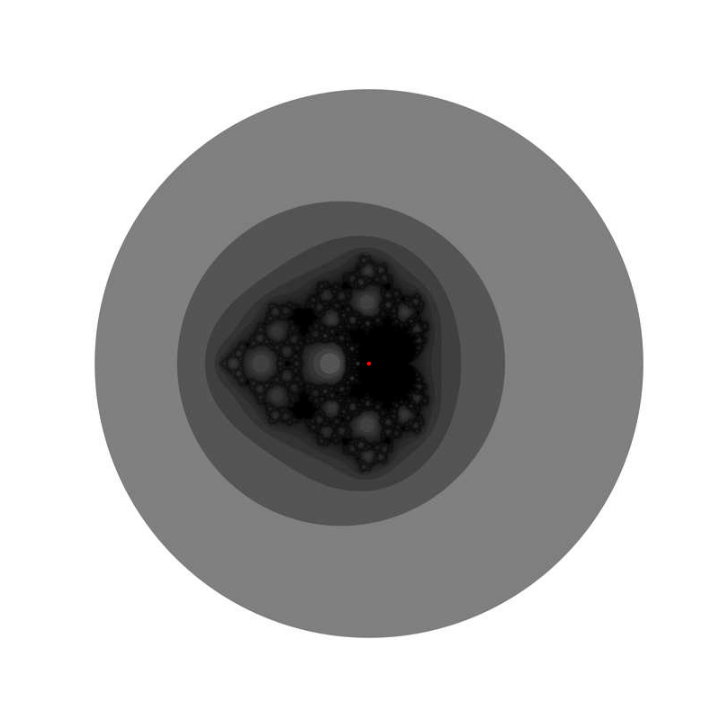}
        \caption{Parameter plane for the family $z\mapsto z^2 + \lambda/z$. We highlight the regions where $\Sigma(R_\lambda)$ is different.}
        \label{fig:ppl21}
    \end{subfigure}

    \begin{subfigure}[!b]{0.45\textwidth}
        \centering
        \includegraphics[width=0.45\textwidth]{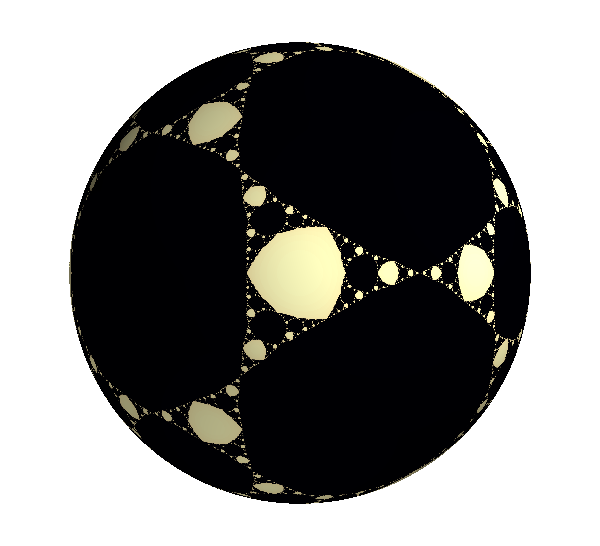}
        \caption{$\lambda = 0.1$}
        \label{fig:js21l01}
    \end{subfigure}
    \begin{subfigure}[!b]{0.45\textwidth}
        \centering
        \includegraphics[width=0.45\textwidth]{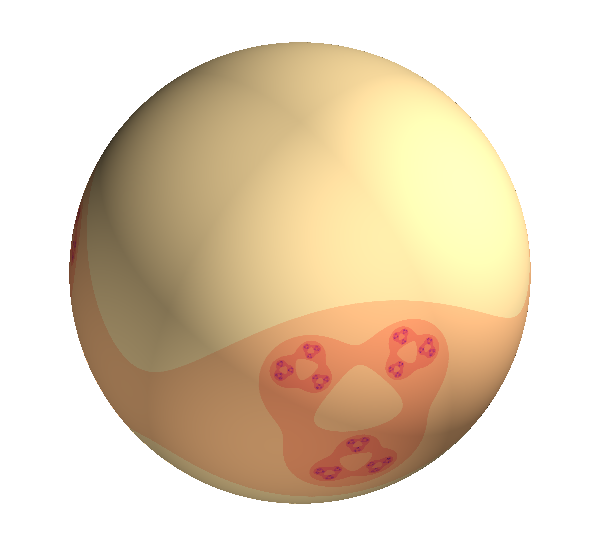}
        \caption{$\lambda = 1$}
        \label{fig:js21l1}
    \end{subfigure}
    \begin{subfigure}[!b]{0.45\textwidth}
        \centering
        \includegraphics[width=0.45\textwidth]{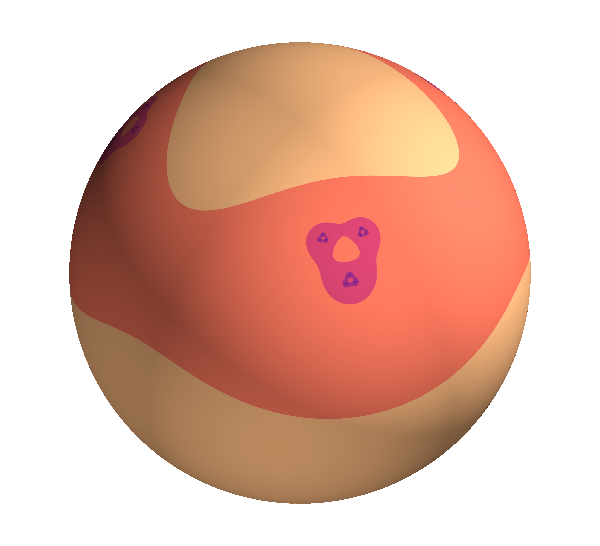}
        \caption{$\lambda = 10$}
        \label{fig:js21l10}
    \end{subfigure}
\end{figure}

\section*{Acknowledgements}
I would like to thank Mitsu Shishikura, Sebastian van Strien, Fedor Pakovich, Sylvain Bonnot and Laura DeMarco for their comments. Last but not least, my heartfelt thanks to Luna Lomonaco, my adviser. I acknowledge financial support from CNPq grant no. 158128/2017-6.


\end{document}